\documentclass[letterpaper, 10 pt, conference]{IEEEconf}  

\IEEEoverridecommandlockouts                              

\usepackage{cite}

\usepackage{amsmath,amssymb,amsfonts}
\usepackage{graphicx}
\usepackage{textcomp}
\usepackage{mathtools}
\usepackage{xcolor}
\usepackage{enumerate}
 
\usepackage{amsthm}
\usepackage{fancyhdr}
\usepackage{algorithm}
\usepackage[]{algpseudocode}
\def\BibTeX{{\rm B\kern-.05em{\sc i\kern-.025em b}\kern-.08em
    T\kern-.1667em\lower.7ex\hbox{E}\kern-.125emX}}
    
\usepackage{epstopdf}

\newtheorem{theorem}{Theorem}

\newtheorem{lemma}{Lemma}
\newtheorem{definition}{Definition}
\newtheorem{remark}{Remark}
\newtheorem{assum}{Assumption}

\title{\LARGE \bf
Stochastic Model Predictive Control for tracking of distributed linear systems with additive uncertainty
}

\author{Christoph Mark and Steven Liu
\thanks{Institute  of  Control  Systems,  Department  of  Electrical  and  Computer
Engineering, University of Kaiserslautern, Erwin-Schrödinger-Str. 12, 67663
Kaiserslautern, Germany, {\tt\small mark|sliu@eit.uni-kl.de}}%
}

\fancypagestyle{FirstPage}{
	\lfoot{		\begin{minipage}{\textwidth}
			\footnotesize
			This version has been accepted for publication in Proc. European Control Conference (ECC), 2021. Personal use of this material is permitted. Permission from EUCA must be obtained for all other uses, in any current or future media, including reprinting/republishing this material for advertising or promotional purposes, creating new collective works, for resale or redistribution to servers or lists, or reuse of any copyrighted component of this work in other works.
	\end{minipage}} 
}
\begin{document}

\maketitle
\thispagestyle{empty}
\pagestyle{empty}

\begin{abstract}
In this paper, we propose a chance constrained stochastic model predictive control scheme for reference tracking of distributed linear time-invariant systems with additive stochastic uncertainty. The chance constraints are reformulated analytically based on mean-variance information, where we design suitable Probabilistic Reachable Sets for constraint tightening. Furthermore, the chance constraints are proven to be satisfied in closed-loop operation. The design of an invariant set for tracking complements the controller and ensures convergence to arbitrary admissible reference points, while a conditional initialization scheme provides the fundamental property of recursive feasibility. The paper closes with a numerical example, highlighting the convergence to changing output references and empirical constraint satisfaction.
\end{abstract}


\section{Introduction}
\thispagestyle{FirstPage}
Model Predictive Control (MPC) is an optimization based control strategy, that uses a model of a dynamical system to predict the system states into the future. The strength of MPC lies in the ability to compute optimal control inputs subject to state and input constraints \cite{rawlings2017model}. In most of the MPC literature the authors consider the regulation task, that is, the system is steered to the origin or an a-priori known setpoint. If this setpoint changes during online operation, the underlying optimization problem may becomes infeasible, which is due to the fact that the MPC for regulation is designed for steady-state operation \cite{ferramosca2013cooperative}. A promising approach that tackles the aforementioned issue is provided by \cite{limon2008mpc}, where convergence of the closed-loop system is guaranteed under any changing output reference.

In the presence of uncertainty, the literature distinguishes between stochastic \cite{mesbah2016stochastic} and robust \cite{mayne2005robust} approaches, where the latter assumes a bounded disturbance, such that constraints can be satisfied robustly. However, if this bound is large, the resulting feasible region of the MPC can be very conservative. Stochastic MPC (SMPC) addresses this issue and utilizes the probability distribution of the disturbance, which allows for a relaxation of the hard constraints to hold as chance constraints, i.e. with a certain probability. SMPC can roughly be separated into \textit{randomized methods} \cite{schildbach2012randomized, muntwiler2020datadriven}, where the stochastic control problem is approximated via a sampling-average-approximation by simulating disturbance scenarios, or \textit{analytical approximation methods} \cite{mark2019distributed, hewing2018stochastic, farina2013probabilistic}, where we utilize the knowledge of the moments and/or probability distribution to reformulate the chance constraints e.g. via concentration inequalities.  

Most of the existing work on MPC is done in a centralized setting \cite{rawlings2017model}, that is, the plant is modeled as a single unit that is controlled by a single controller. However, if the plant represents a large-scale network of dynamical systems, then centralized MPCs quickly become intractable. To resolve this issue, the control task is distributed over several agents, which leads to distributed MPC (DMPC) \cite{christofides2013distributed}.

\subsubsection*{Similar work}
The authors of \cite{conte2013cooperative} propose a nominal DMPC for tracking based on the concept of distributed invariance. A distributed terminal set for tracking is used to ensure recursive feasibility of the DMPC. In \cite{ferramosca2013cooperative} the authors propose a sequential nominal DMPC that combines the steady-state optimization (reference governor) and the MPC problem, such that only one optimization needs to be solved online. The authors of \cite{farina2017stochastic} propose a distributed SMPC for independent systems, where the chance constraints are approximated with Cantelli's inequality. Chance constraint satisfaction is only guaranteed in prediction.

\subsubsection*{Contribution}
In this paper we propose a distributed SMPC for tracking of dynamically coupled linear systems with additive stochastic uncertainty. We use an expected value cost function for tracking, which is then analytically reformulated by means of mean and covariance of the state and input sequences. The nominal steady-state optimization problem is included through the cost function, which was similarly done in \cite{limon2008mpc, farina2017stochastic, conte2013cooperative}. Based on the predicted covariance sequences, we compute suitable Probabilistic Reachable Sets (PRS) for constraint tightening, which render the resulting MPC optimization problem as a deterministic quadratic program. In Theorem \ref{thm:main_result} we provide our main result on recursive feasibility, closed-loop chance constraint satisfaction and convergence to arbitrary admissible output references.
\subsubsection*{Outline}
The paper is organized as follows. In Section \ref{sec:preliminaries} we introduce the notation and the system dynamics. In Section \ref{sec:centralSMPC} we define an affine tube controller to treat the stochasticity of the dynamics and characterize nominal steady-states. Furthermore, we formulate the cost function for tracking and reformulate the chance constraints via mean-variance Probabilistic Reachable Sets. Lastly, we introduce the terminal set for tracking, the conditional initialization scheme and the MPC optimization problem. In Section \ref{sec:distributed_synthesis} we give remarks on how to synthesize all controller ingredients distributedly and how to set up the DMPC, while Section \ref{sec:example} is dedicated to a numerical example. The paper closes with some concluding remarks. For the sake of readability, the proofs are delayed in the appendix.
\section{Preliminaries and problem statement}
\label{sec:preliminaries}
\subsection{Notations}
Given two polytopic sets $\mathbb{A}$ and $\mathbb{B}$, the Pontryagin difference is given as $\mathbb{A} \ominus \mathbb{B} = \{ a \in \mathbb{A} : a+b \in \mathbb{A}, \forall b \in \mathbb{B} \}$. Positive definite and semidefinite matrices are indicated as $A>0$ and $A\geq0$, respectively. Given a matrix $A$ and vector $x$, we denote the $j$-th row of $A$ as $[A]_j$, the $i$-th element of the $j$-th row as $[A]_{j,i}$ and the $j$-th element of $x$ as $[x]_j$. The spectral radius of a matrix $A$ is denoted as $\rho(A)$. The weighted 2-norm is $\Vert x \Vert_P = \sqrt{x^\top P x}$.
For an event $E$ we define the probability of occurrence as $\mathbb{P}(E)$, whereas the expected value of a random variable $w$ is given by $\mathbb{E} \{w \}$. Two random variables $x$, $y$ that share
the same distribution are equal in distribution, denoted by $x \overset{d}{=} y$. The set $ \{ 1, ..., M \} \subseteq \mathbb{N} $ is denoted as $\mathcal{M}$. For two vectors $x_1 \in \mathbb{R}^{n_1}$ and $x_2 \in \mathbb{R}^{n_2}$ we denote the stacked vector as $x = \text{col}_{j \in \{1, 2\}} (x_j) \in \mathbb{R}^{n_1 + n_2}$.
\subsection{Stochastic dynamics and Chance Constraints}
In this work we consider a network of $M$ linear time-invariant systems
\begin{subequations}
\begin{align}
	x_i(k+1) &= \sum_{j = 1}^M A_{ij} x_j(k) + B_i u_i(k) + w_i(k) \label{eq:dynamics}\\ 
	y_i(k) &= \sum_{j=1}^M C_{ij} x_j(k) \label{eq:true_output}
\end{align}
\label{eq:real_system}
\end{subequations}
where $x_i \in \mathbb{R}^{n_i}$, $u_i \in \mathbb{R}^{m_i}$ and $y_i \in \mathbb{R}^{l_i}$ denote the state, input and output vectors. We assume that $w_i \in \mathbb{R}^{n_i}$ is a zero-mean random variable that is distributed according to $w_i \sim \mathcal{Q}^{w_i}(0, \Sigma^w_i)$ with known covariance matrix $\Sigma^w_i > 0$. 
\begin{assum}
The distribution $\mathcal{Q}^{w_i}$ is central convex unimodal (CCU) \cite{dharmadhikari1976multivariate} for all $i \in \mathcal{M}$.
\end{assum}
 To ease the notation, we define the dynamic neighborhood of each subsystem.
\begin{definition}[Dynamic neighborhood]
System $j$ is a neighbor of system $i$ if $A_{ij} \neq 0$ and/or $C_{ij} \neq 0$. The set of all neighbors of system $i$, including system $i$ itself, is denoted as $\mathcal{N}_i$. The states of all systems $j \in \mathcal{N}_i$ are denoted as $x_{\mathcal{N}_i} \in \text{col}_{j \in \mathcal{N}_i}(x_j) \in \mathbb{R}^{n_{\mathcal{N}_i}}$.  \label{def:neighboring_systems}
\end{definition}
Thus, the local dynamics \eqref{eq:real_system} are expressed compactly as
\begin{align*}
		x_i(k+1) &= A_{\mathcal{N}_i} x_{\mathcal{N}_i}(k) + B_i u_i(k) + w_i(k)  \\
		y_i(k) &= C_{\mathcal{N}_i} x_{\mathcal{N}_i}(k).
\end{align*}
Furthermore, we impose polytopic state and input chance constraints for any $k \geq 0$
\begin{subequations}
\begin{align}
	&\mathbb{P} \big (  x_{\mathcal{N}_i}(k) \in \mathbb{X}_{\mathcal{N}_i} \coloneqq \{ x_{\mathcal{N}_i} \: | \: H_{\mathcal{N}_i} x_{\mathcal{N}_i} \leq h_{\mathcal{N}_i} \}  \big ) \geq p_{x,i} \label{eq:local_chance_constraints:state}\\
	&\mathbb{P} \big (  u_i(k) \in \mathbb{U}_i \coloneqq \{ u_i \: | \: L_i u_i \leq l_i \} \big ) \geq p_{u,i},
\end{align}
\label{eq:local_chance_constraints}
\end{subequations}
where $h_{\mathcal{N}_i} \in \mathbb{R}^{p_i}$, $l_i \in \mathbb{R}^{q_i}$ and $p_{x,i}, p_{u,i} \in (0,1)$ are the levels of chance constraint satisfaction. In this formulation we can impose local and neighbor-to-neighbor coupled state chance constraints, as well as local input chance constraints. By combining the local dynamics \eqref{eq:real_system} for all $i \in \mathcal{M}$, we obtain the global system
\begin{subequations}
\begin{align}
	x(k+1) &= A x(k) + B u(k) + w(k) \\
	y(k) &= C x(k),  \label{eq:global_system:output}
\end{align}
\label{eq:global_dynamics}
\end{subequations}
where $x \in \mathbb{R}^{n}$, $u \in \mathbb{R}^{m}$, $y \in \mathbb{R}^{l}$, $w \sim \mathcal{Q}(0, \Sigma^w)$ and $\Sigma^w = \text{diag}(\Sigma^w_1, \ldots, \Sigma^w_M)$. We make the following assumption on stabilizability.
\begin{assum}
The pair $(A,B)$ is stabilizable with a structured linear feedback controller
\begin{align*}
\pi(x(k) ) = Kx(k) = \text{col}_{i \in \mathcal{M}}(K_{\mathcal{N}_i} x_{\mathcal{N}_i}(k)),
\end{align*}	
where $K_{\mathcal{N}_i} \in \mathbb{R}^{m_i \times n_{\mathcal{N}_i}}$, such that $\rho(A + BK) < 1$.
\label{assum_stabilizable}
\end{assum}
\begin{remark}
A controller that satisfies Assumption \ref{assum_stabilizable} can be found with \cite[Lemma 10 and Proposition 13]{conte2016distributed}.
\label{rem:stabilizing_controller}
\end{remark}
The Cartesian product of the local constraint sets \eqref{eq:local_chance_constraints} gives us a global representation
\begin{align*}
&\mathbb{X} \coloneqq \{ x_{\mathcal{N}_i} \: | \: H_{\mathcal{N}_i} x_{\mathcal{N}_i} \leq h_{\mathcal{N}_i}, \forall i \in \mathcal{M}\} = \{x \: | \: H x \leq h \}\\
&\mathbb{U} \coloneqq \{ u_i \: | \: L_i u_i \leq l_i , \forall i \in \mathcal{M}\} = \{ u \: | \: L u \leq l \},
\end{align*}
where we make the following assumption:
\begin{assum}
The sets $\mathbb{X}$ and $\mathbb{U}$ are compact.
\label{assum:compact}
\end{assum}
\begin{remark}
\label{rem:compactness}
The compactness of the constraint sets is mandatory for the computation of an ellipsoidal invariant set for tracking, see Section \ref{sec:tracking_set}. If some states or inputs are unconstrained, it is always possible to constrain them with a large, but finite value, such that Assumption \ref{assum:compact} holds.
\end{remark}

The communication graph $\mathcal{G}(\mathcal{V},\mathcal{E})$ of the distributed MPC is induced by the dynamic couplings, where each vertex in $\mathcal{V}$ corresponds to a subsystems $i \in \mathcal{M}$. The edges $\mathcal{E}$ represent the connection between the subsystems according to Definition \ref{def:neighboring_systems}, i.e. $\mathcal{N}_i = \{j | (i,j) \in \mathcal{E} \} \cup \{ i \} $.
\begin{assum}
The communication graph $\mathcal{G}(\mathcal{V},\mathcal{E})$ is bidirectional, i.e. if $j \in \mathcal{N}_i$, then $i \in \mathcal{N}_j$.
\end{assum}
\section{Stochastic MPC for reference tracking}
\label{sec:centralSMPC}
\subsection{Controller structure}
In this work we follow standard procedures in stochastic MPC \cite{mark2020stochastic} and define for each subsystem $i \in \mathcal{M}$ a distributed tube controller according to Assumption \ref{assum_stabilizable}, i.e.
\begin{align}
	u_i(k) = v_i(0|k) + K_{\mathcal{N}_i} ( x_{\mathcal{N}_i}(k) - z_{\mathcal{N}_i}(0|k) ), \label{eq:tube_controller}
\end{align}
where the nominal states and inputs $z_i$, $v_i$ are governed by the dynamic equations 
\begin{subequations}
\begin{align}
	z_i(t+1|k) &= A_{\mathcal{N}_i} z_{\mathcal{N}_i}(t|k) + B_i v_i(t|k) \label{eq:nominal_dynamics} \\
	\bar{y}_i(t|k) &= C_{\mathcal{N}_i} z_{\mathcal{N}_i}(t|k). \label{eq:nominal_output}
\end{align}
\label{eq:local_nominal_dynamics}
\end{subequations}
The nominal input sequences $v_i(t|k)$ for $t \in \{0, \ldots, N-1\}$ are obtained from an MPC optimization problem solved at time step $k$ and $z_i(t|k)$ is the resulting $t$-step ahead prediction of the states. 
Define the error state as $e = x - z$, then it can be shown that the prediction error evolves linearly as
\begin{align}
	e_i(t+1|k) = A_{\mathcal{N}_i,K} e_{\mathcal{N}_i}(t|k) + w_i(t|k), \label{eq:error_state}
\end{align}
where $A_{\mathcal{N}_i,K} = A_{\mathcal{N}_i} + B_i K_{\mathcal{N}_i}$ and $w_i(t|k) \overset{d}{=} w_i(t+k)$.
The corresponding global dynamics are given by
\begin{subequations}
\begin{align}
	z(t+1|k) &= A z(t|k) + B v(t|k) \label{eq:global:nominal_dynamics}\\
	e(t+1|k) &= A_K e(t|k) + w(t|k) \label{eq:global:error} \\
	\bar{y}(t|k) &= C z(t|k). \label{eq:global:expected_output}
\end{align}
\end{subequations}
\subsection{Steady-states}
In order to define a tracking objective we need to characterize the steady-states of the system \eqref{eq:global_dynamics}. However, the persistent exogenous disturbance $w$ only allows for a formulation of a steady-state in expectation, i.e. w.r.t. the nominal dynamics \eqref{eq:global:nominal_dynamics}. The corresponding steady-state condition is given by
\begin{align*}
	z_s(k) = A z_s(k) + B v_s(k),  
\end{align*}
where $(z_s(k),v_s(k))$ denotes the steady-state pair that is consistent with the output $y_s(k) = C z_s(k)$. By combining the above equations we can write the steady-state condition compactly as
\begin{align}
	\begin{bmatrix}
		A-I & B \\
		C 	& 0
	\end{bmatrix} \begin{bmatrix}
		z_s(k) \\	
		v_s(k)
	\end{bmatrix} = 
	\begin{bmatrix}
		0  \\
		y_s(k)
	\end{bmatrix}. \label{eq:steady_state_condition}
	\end{align}
\begin{remark}
The artificial tracking target $y_s(k)$ replaces the actual reference $y_{ref}$. The reason behind this is to allow the MPC to operate with shorter prediction horizon by tracking at each time-step $k$ a $N$-step reachable reference $y_s(k)$ \cite{limon2008mpc}. By introducing a tracking objective, we can steer the artificial tracking target $y_s(k)$ to $y_{ref}$ in an admissible way. With the time dependency of $y_s(k)$ we stress that at each closed-loop time instant $k$, the artificial steady-state pair $(z_s(k),v_s(k))$ is variable.
\end{remark}
\subsection{Objective function}
The objective in stochastic MPC for reference tracking is to steer the output \eqref{eq:global_system:output} in expectation to the desired reference $y_{ref}$. To this end, define the deviation variables $\Delta x = x - z_s$ and $\Delta u = u - v_s$ and the cost function for tracking
\begin{align*}
	J &= J_{MPC}(\Delta x, \Delta u) + J_o(y_s, y_{ref}),
\end{align*}
where the MPC cost is 
\begin{align}
J_{MPC}(\Delta x, \Delta u) &= \mathbb{E} \bigg \{ \Vert \Delta x(N|k)\Vert_P^2 \nonumber  \\
&+  \sum_{t=0}^{N-1} \Vert \Delta x(t|k)\Vert_Q^2 + \Vert \Delta u(t|k) \Vert_R^2 \bigg \} \label{eq:expected_cost}
\end{align}
and the tracking cost is $J_o(y_s, y_{ref}) = \Vert y_s(k) - y_{ref} \Vert_T^2$. The matrices $Q, R, T$ are assumed to be block diagonal weighting matrices for the state, input and output residuals. The block diagonal matrix $P$ is the solution to
\begin{align}
 A_K^\top P A_K - P = -Q - K^\top R K \label{eq:lyapunov_eqn}
\end{align}
with $A_K = A + BK$, which by Assumption \ref{assum_stabilizable} is guaranteed to exist. 
\subsection*{Analytic evaluation}
The MPC cost $J_{MPC}$ can be further evaluated analytically by substituting $\Delta x = z + e - z_s$ and $\Delta u = v + K e - z_s$, i.e.
\begin{align}
	&J_{MPC} = \Vert \Delta z(N|k)\Vert_P^2 + \sum_{t=0}^{N-1} \Vert \Delta z(t|k)\Vert_Q^2 + \Vert \Delta v(t|k) \Vert_R^2 \nonumber \\
	&+ \text{tr}\{ P \Sigma^e(N|k) \} +  \sum_{t=0}^{N-1} \text{tr}\{ (Q + K^\top R K) \Sigma^e(t|k) \}, \label{eq:variance_cost}
\end{align}
where the first line represents the mean part with $\Delta z = \mathbb{E}\{\Delta x \} =  z - z_s$ and $\Delta v = \mathbb{E}\{\Delta u \} = v - v_s$. The second line corresponds to the variance part of the cost, where the covariance sequence $\Sigma^e(t+1|k) = \mathbb{E}\{ e(t+1|k) e(t+1|k)^\top \}$ is obtained from \eqref{eq:global:error}, that is
\begin{align}
	\Sigma^e(t+1|k) = A_K \Sigma^e(t|k) A_K^\top + \Sigma^w. \label{eq:global_covariance}
\end{align}
Note that \eqref{eq:global_covariance} does not depend on the MPC optimization variables $z,v$. Hence, the sequence \eqref{eq:global_covariance} can be computed offline and the variance part can be neglected in the receding horizon cost function
\begin{align}
J &= \Vert \Delta z(N|k)\Vert_P^2 + \sum_{t=0}^{N-1} \Vert \Delta z(t|k)\Vert_Q^2 + \Vert \Delta v(t|k) \Vert_R^2 \nonumber \\
&+ \Vert y_s(k) - y_{ref} \Vert_T^2 \label{eq:nominal_cost_function}.
\end{align}
\subsection{Chance constraint reformulation}
In order to address the chance constraints we make use of PRS for constraint tightening. 
\begin{definition}[Probabilistic $t$-step Reachable Set]
\label{def:t_step_PRS}
A set $\mathcal{R}_t$ with $t \geq 0$ is said to be a $t$-step PRS of probability level $p$ for system \eqref{eq:global:error} if $e(0|k) = 0 \Rightarrow \mathbb{P}( e(t|k) \in \mathcal{R}_t) \geq p$.
\end{definition}
\begin{definition}[Probabilistic Reachable Set]
	\label{def:PRS}
	A set $\mathcal{R}$ is said to be a PRS of probability level $p$ for system \eqref{eq:error_state} if
	\begin{align*}
		e(0|k) = 0 \Rightarrow \mathbb{P}( e(t|k) \in \mathcal{R}) \geq p \quad \forall t \geq 0
	\end{align*}
\end{definition}
In this paper we follow the lines of \cite{mark2020stochastic, hewing2018stochastic} and use a $t$-step mean-variance PRS for the predicted error \eqref{eq:global:error}. Note that due to \eqref{eq:global:error} and $\mathbb{E}\{ w \}=0$ the error $e$ is zero-mean, thus, the PRS is fully characterized by the error covariance \eqref{eq:global_covariance}. By applying Chebyshev's inequality \cite[Thm. 1]{chen2007new} along each dimension $j = 1, \ldots, n$ of $e$, we obtain a deterministic expression for the $t$-step PRS
\begin{align}
	\mathcal{R}_t = \bigg \{ [e]_j\: \bigg| \: \big| [e]_j \big| \leq  \sqrt{\gamma \:[\Sigma^e(t|k)]_{j,j}}, \: \forall j = 1, \ldots, n \bigg \},  \label{eq:PRS:state}
\end{align}
where $\gamma = n / (1 - p_x)$. The input PRS $R_t^u(k)$ is defined analogously by using the input error $e^u = u - v = K e$ with var$(e^u) = \Sigma^u(t|k) = K \Sigma^e(t|k) K^\top$ and $\gamma^u = m / (1 - p_u)$. Note that the probability levels $p_x$ and $p_u$ are used to regulate the levels of chance constraint satisfaction in \eqref{eq:local_chance_constraints}.
\begin{remark}
The bound $\gamma$ holds for arbitrary probability distributions. However, if the disturbance is normally distributed, then $\gamma = \mathcal{X}_{n}^2(p_x)$ yields the tightest probability bound, where $\mathcal{X}_{n}^2(p_x)$ is the quantile function of the Chi-squared distribution at probability level $p_x$ with $n$ degrees of freedom. This similarly holds for $\gamma^u$. \label{remark:quantile}
\end{remark}
Similar to the $t$-step PRS we define a PRS as
\begin{align}
	\mathcal{R} = \bigg \{ [e]_j\: \bigg| \: \big| [e]_j \big| \leq  \sqrt{\gamma \:[\Sigma^e_f]_{j,j}}, \: \forall j = 1, \ldots, n \bigg \}, \label{eq:PRS_parallel}
\end{align}
where $\Sigma^e_f$ is the steady-state covariance matrix that satisfies the Lyapunov equality
\begin{align}
	\Sigma_f^e = A_K \Sigma_f^e A_K^\top + \Sigma^w. \label{eq:steady_state_covariance_global}
\end{align}
The existence of the solution $\Sigma_f^e$ is guaranteed by Schur stability of $A_K$ and $\Sigma^w > 0$, which implies that the covariance sequence \eqref{eq:global_covariance} converges to $\Sigma_f^e$ as $t\rightarrow \infty$.
By tightening the constraint sets $\mathbb{X}$ and $\mathbb{U}$ with the PRS $\mathcal{R}_t$ and $\mathcal{R}_t^u$, it can easily be verified that the chance constraints \eqref{eq:local_chance_constraints} are satisfied in prediction if the nominal states and inputs satisfy the following conditions
\begin{subequations}
	\begin{align}
		&z(t|k) \in \mathbb{Z}_t \coloneqq \mathbb{X} \ominus \mathcal{R}_t \quad \forall t = 0, \ldots, N-1 \label{eq:global_nominal_constraints:state} \\
		&v(t|k) \in\mathbb{V}_t \coloneqq \mathbb{U} \ominus \mathcal{R}^u_t \quad \forall t = 0, \ldots, N-1. \label{eq:global_nominal_constraints:input}
	\end{align}
	\label{eq:global_nominal_constraints}
	\end{subequations}
\begin{remark}
	The t-step predictive covariance sequence $\Sigma^e(t|k)$ in \eqref{eq:PRS:state} is defined for each closed-loop time step $k$. However, since the distribution $\mathcal{Q}(0, \Sigma^w)$ is time invariant, it holds that $\Sigma^e(t|0) = \Sigma^e(t|1) = \ldots = \Sigma^e(t|k)$ for all $k\geq 0$. The resulting $t$-step PRS is therefore constant for each $k$.
\end{remark}
\subsection{Terminal set for tracking}
\label{sec:tracking_set}
In this section we define, as introduced by \cite{limon2008mpc}, a terminal set for tracking in the augmented state $\alpha = [ \Delta z^\top, \: z_s^\top, \: v_s^\top]^\top$ together with the augmented dynamics
\begin{align}
	\alpha(k+1) = \underbrace{\begin{bmatrix}
		A & 0 & 0\\
		0 & I & 0\\
		0 & 0 & I
	\end{bmatrix}}_{A_e} \alpha(k) + \underbrace{\begin{bmatrix}
	B \\ 
	0 \\
	0 
	\end{bmatrix}}_{B_e} \Delta v(k). \label{eq:augmented_system}
\end{align}
\begin{definition}
\label{def:invariant_set}
Consider the control law $\kappa_f(\Delta z) = K \Delta z$ and the augmented system $\alpha(k+1) = A_e \alpha(k) + B_e K \Delta z(k)$. The set $\mathbb{Z}_{tr} \subseteq \mathbb{R}^{2n + m}$ is an admissible invariant set for tracking if for all $\alpha \in \mathbb{Z}_{tr}$ it holds that:
\begin{subequations}
\begin{align}
	&A_e \alpha + B_e K \Delta z \in \mathbb{Z}_{tr} \label{eq:terminal_set:invariance_cond}  \\
	&\Delta z + z_s \in \mathbb{Z}_f \coloneqq \mathbb{X} \ominus \mathcal{R}, \quad K \Delta z + v_s \in \mathbb{V}_f \coloneqq \mathbb{U} \ominus \mathcal{R}^u, \label{eq:terminal_set:constr}
\end{align}
\end{subequations}
where $\mathcal{R}, \mathcal{R}^u$ are PRS according to Def. \ref{def:PRS}.
\label{def:terminal_set}
\end{definition}

\subsection{Initial constraints}
Ideally, we want to use always the latest state information $z(0|k) = x(k)$ to initialize the MPC optimization problem, which we call Mode $1$. However, the unboundedness of $w$ may renders the initial state infeasible in \eqref{eq:global_nominal_constraints:state}. To ensure the fundamental property of recursive feasibility we require a backup strategy (Mode $2$), where we utilize the shifted optimal solution from the previous time step $z(0|k) = z(1|k-1)$. The constraint is implemented as
 \begin{align}
 	z(0|k) = \{x(k), z(1|k-1) \}. \label{eq:initial_value}
 \end{align}
\subsection{MPC optimization problem}
The following MPC optimization problem is solved at every time instant $k \geq 0$ 
\begin{subequations}
\label{mpc_problem}
\begin{alignat}{2}
&\!\min_{\mathcal{Z},\mathcal{V},y_s}  &      &   \quad \eqref{eq:nominal_cost_function}  \\
&\text{s.t.} &      & \quad \eqref{eq:global:nominal_dynamics}, \eqref{eq:global_nominal_constraints:state}, \eqref{eq:global_nominal_constraints:input} \quad t=0,...,N-1 \nonumber\\
&                  &      & \quad \eqref{eq:steady_state_condition}, \eqref{eq:initial_value} \\
&                  &      & \quad (z(N|k) - z_s(k), z_s(k), v_s(k)) \in  \mathbb{Z}_{tr}, \label{eq:terminal_constraint}
\end{alignat}
\end{subequations}
where $\mathcal{V} = \{ v(0|k), \ldots, v(N-1|k) \}$ and $\mathcal{Z} = \{ z(0|k), \ldots, z(N|k) \}$ denote the input and state sequences.
\section{Distributed synthesis}
\label{sec:distributed_synthesis}
This section is dedicated to the distributed synthesis of distributed PRS and a distributed invariant set for tracking. For the resulting distributed online algorithm we need to ensure that the following quantities are structured:
\begin{enumerate}[A)]
	\item Dynamics \eqref{eq:local_nominal_dynamics} and steady-state condition \eqref{eq:steady_state_condition}
	\item Cost function \eqref{eq:nominal_cost_function} 
	\item Constraints \eqref{eq:global_nominal_constraints}
	\item Terminal set (Def. \ref{def:terminal_set})
\end{enumerate}
\subsection{Dynamics}
The dynamics \eqref{eq:local_nominal_dynamics} are structured by design  \eqref{eq:global:nominal_dynamics}. Similarly, we can decompose the central steady-state condition \eqref{eq:steady_state_condition}, such that each subsystem $i \in \mathcal{M}$ only requires information from its neighbors $j \in \mathcal{N}_i$.
\subsection{Cost function}
The matrices $Q$, $R$ and $T$ are assumed to be block diagonal. Assumption \ref{assum_stabilizable} ensures the existence of a structured tube-controller $K_{N_i}$, which in turn implies the existence of a block diagonal matrix $P = \text{diag}(P_1, \ldots, P_M)$ that satisfies \eqref{eq:lyapunov_eqn}. Thus, the cost function \eqref{eq:nominal_cost_function} is block separable.
\subsection{Constraints} 
\label{sec:synthesis:constraints}
Since the constraint sets \eqref{eq:local_chance_constraints} are structured by definition, it remains to find distributed PRS for the states and inputs. The design ultimately reduces to the computation of block diagonal upper bounds of the covariance sequence \eqref{eq:global_covariance}, i.e.
\begin{align}
\Sigma^e(t|k) \leq \hat{\Sigma}^e(t|k) = \text{diag}(\hat{\Sigma}_{1}^e(t|k), \ldots, \hat{\Sigma}_{M}^e(t|k)) \label{eq:block_diagonal_bound}
\end{align}
for all $t = 0, \ldots, N-1$. The design of such matrices can be carried out via distributed semidefinte programming (SDP) \cite{mark2020stochastic} or via iterative methods \cite{farina2016distributed, mark2019distributed}, where a modified version of the local covariance update equation
\begin{align*}
	\Sigma_i^e(t+1|k) = A_{\mathcal{N}_i,K} \Sigma^e_{\mathcal{N}_i}(t|k) A_{\mathcal{N}_i,K}^\top + \Sigma_i^w
\end{align*} 
is utilized. In this paper we chose the latter approach and define
\begin{align}
	\hat{\Sigma}_i^e(t+1|k) = \tilde{A}_{\mathcal{N}_i,K} \hat{\Sigma}^e_{\mathcal{N}_i}(t|k) \tilde{A}_{\mathcal{N}_i,K}^\top + \Sigma_i^w, \label{eq:modified_update_eq}
\end{align}
where $\tilde{A}_{\mathcal{N}_i,K} = \sqrt{|\mathcal{N}_i|} A_{\mathcal{N}_i,K}$ and $|\mathcal{N}_i|$ denotes the cardinality of the set of neighbors according to Def. \ref{def:neighboring_systems}. From \cite[Lem. 2]{farina2016distributed} we have that, if $\Sigma^e(t|k) \leq \hat{\Sigma}^e(t|k)$ and $ \hat{\Sigma}^e(t+1|k)$ is updated according to \eqref{eq:modified_update_eq}, then it holds that $\Sigma^e(t+1|k) \leq \hat{\Sigma}^e(t+1|k)$. Similarly we can upper bound the steady-state covariance \eqref{eq:steady_state_covariance_global} with $\Sigma_{f}^e \leq \hat{\Sigma}_{f}^e $.

Distributed PRS for the terminal set, as well as distributed $t$-step PRS for the state and input constraints can now analogously to \eqref{eq:PRS:state} and \eqref{eq:PRS_parallel} be obtained via $\hat{\Sigma}_{f,\mathcal{N}_i}$ and $\hat{\Sigma}_{\mathcal{N}_i}(t|k)$.
\begin{align*}
	\mathbb{Z}_{f, \mathcal{N}_i} &\coloneqq \mathbb{X}_{\mathcal{N}_i} \ominus \mathcal{R}_{{ \mathcal{N}_i}}, \quad \mathbb{V}_{f,i} \coloneqq \mathbb{U}_i \ominus \mathcal{R}^u_{i}\\
\mathbb{Z}_{t, \mathcal{N}_i} &\coloneqq \mathbb{X}_{\mathcal{N}_i} \ominus \mathcal{R}_{t,{\mathcal{N}_i}}, \quad \mathbb{V}_{t,i} \coloneqq \mathbb{U}_i \ominus \mathcal{R}^u_{t,i}.
\end{align*}
\subsection{Terminal set}
We follow the lines of \cite{conte2013cooperative} and define an ellipsoidal invariant set for tracking for the augmented system \eqref{eq:augmented_system}. Let $V_{tr}(\alpha)$ be a quadratic cost function 
\begin{align*}
	V_{tr}(\alpha) = \alpha^\top \begin{bmatrix}
	P_f & 0 & 0 \\
	0 & P_z & 0 \\
	0 & 0 & P_v
	\end{bmatrix} \alpha, 
\end{align*}
where $P_f \in \mathbb{R}^{n \times n}$, $P_z \in \mathbb{R}^{n \times n}$ and $P_v \in \mathbb{R}^{m \times m}$. Due to its block diagonal structure, this allows for a separation 
\begin{align*}
	V_{tr}(\alpha) = \sum_{i=1}^M \Delta z_i^\top P_{f,i} \Delta z_i + z_s^\top P_{z,i} z_{s,i} + v_{s,i}^\top P_{v,i} v_{s,i}.
\end{align*}
By Assumption \ref{assum_stabilizable} there exists a structured terminal controller $\Delta v = K \Delta z$, which renders any level set of
\begin{align*}
	\mathbb{Z}_{tr} = \{ \alpha \in \mathbb{R}^{2n + m} \: | \: V_{tr}(\alpha) \leq 1 \}
\end{align*}
invariant under the augmented dynamics \eqref{eq:augmented_system}. The matrices $P_{f,i}, P_{z,i}$ and $P_{v,i}$ can be synthesized for all systems $i \in \mathcal{M}$ distributedly via distributed optimization, e.g. with \cite[Thm. IV.2]{conte2013cooperative}, where the local PRS for terminal constraint tightening \eqref{eq:terminal_set:constr} are obtained from the procedure in Section \ref{sec:synthesis:constraints}. Note that due to Assumption \ref{assum:compact} the tracking ellipsoids $P_z$ and $P_v$ are finite. The invariance property of the terminal set for tracking can then be enforced distributedly, e.g. as the authors of \cite{conte2016distributed} have demonstrated.
\subsection{Distributed optimization based MPC}
In this paper, we use the Alternating direction method of multipliers (ADMM) \cite{boyd2011distributed} to solve the optimization problem \eqref{mpc_problem}. In \cite{conte2013robust}, the authors provided a corresponding formulation for distributed MPC, which we adopted for the numerical example. 
\subsection{Main result}
Before stating the main result, we need to characterize the set of admissible nominal steady-states and outputs that are constrained by \eqref{eq:global_nominal_constraints}.
\begin{definition}
The set of admissible nominal steady-states is given by
\begin{align*}
	\mathbb{Z}_s = \{ (z_s, u_s) \in (\mathbb{Z}_f \times \mathbb{V}_f) \: | \: (A - I) z_s + B v_s = 0 \}.
\end{align*}
The set of admissible outputs is given by
\begin{align*}
	\bar{\mathbb{Y}}_s = \{ C z_s \: | \: z_s \in \mathbb{Z}_s \}.
\end{align*}
\end{definition}
\begin{theorem}
\label{thm:main_result}
If the MPC optimization Problem \ref{mpc_problem} admits a feasible solution at time $k=0$, then it is recursively feasible and the chance constraints \eqref{eq:local_chance_constraints} are satisfied in closed-loop for any $k \geq 0$. Furthermore, if $y_{ref} \in \bar{\mathbb{Y}}_s$, then  
\begin{align*}
	\lim_{k \rightarrow \infty} \mathbb{E} \{ {y}(k) \} = y_{ref}.
\end{align*}
If $y_{ref} \notin \bar{\mathbb{Y}}_s$, then $\mathbb{E} \{ y(k) \}$ converges to the admissible output $\tilde{y}_s$ that minimizes the tracking cost
\begin{align*}
	\tilde{y}_s = \!\arg \: \!\min_{y_s \in \bar{\mathbb{Y}}_s} \Vert y_s - y_{ref} \Vert_T^2
\end{align*}
\end{theorem}

\section{Numerical example}
\label{sec:example}
 \begin{figure*}[htbp]
\centering
\begin{minipage}{.33\textwidth}
  \centering
  \includegraphics[width=0.97\linewidth]{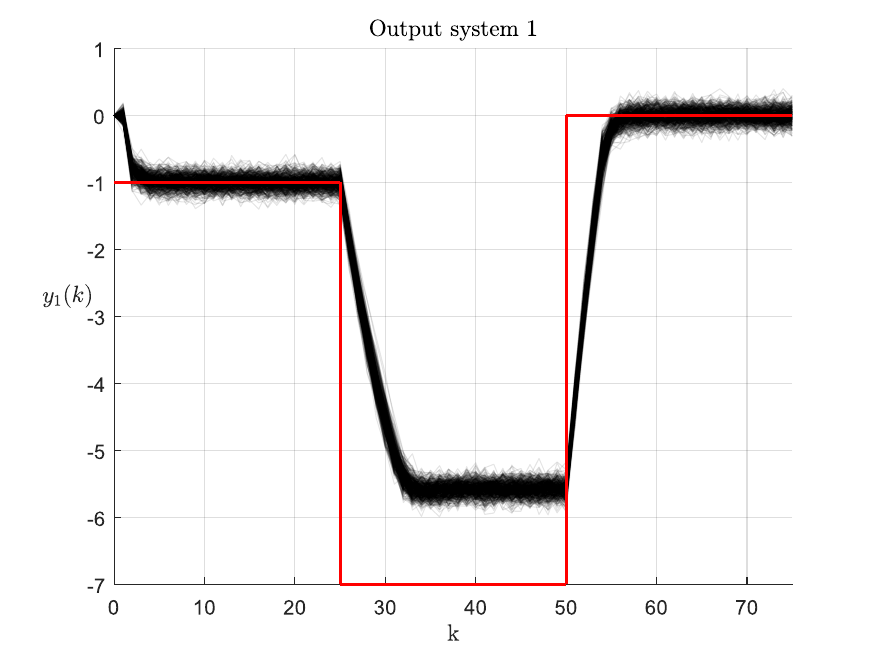}
\end{minipage}%
\begin{minipage}{.33\textwidth}
  \centering
  \includegraphics[width=0.97\linewidth]{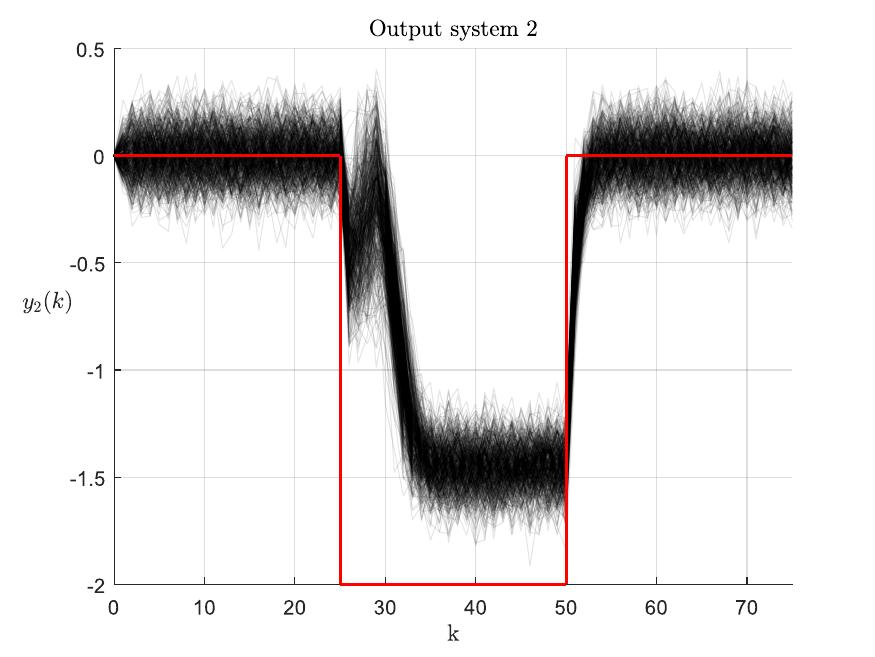}
\end{minipage}
\begin{minipage}{.33\textwidth}
  \centering
  \includegraphics[width=0.97\linewidth]{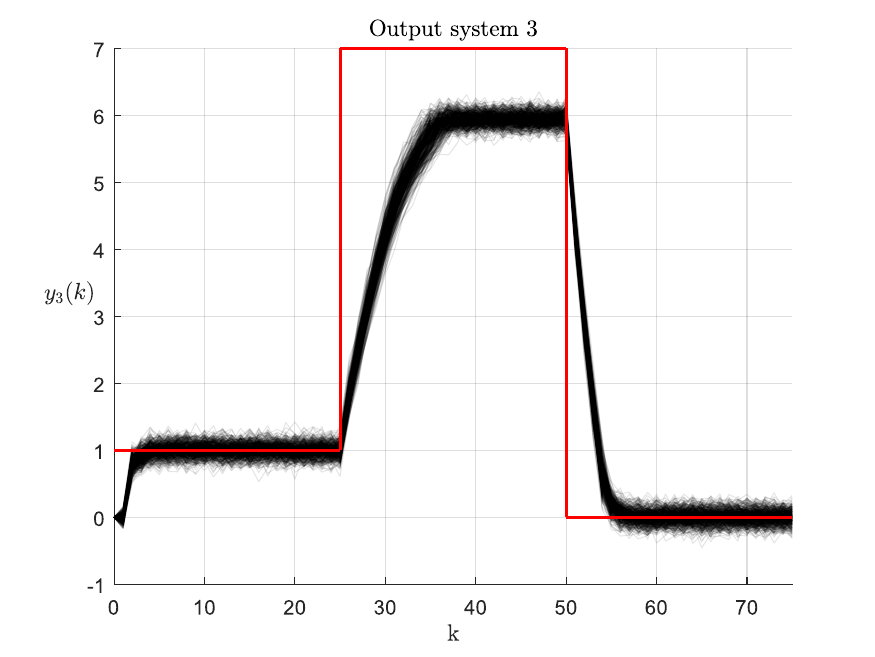}
\end{minipage}
\caption{$1000$ closed-loop output trajectories (black) and the corresponding reference signal (red).}
\label{fig:output}
\end{figure*}
This section is dedicated to a numerical example. We consider $M=3$ subsystems with neighbors $\mathcal{N}_i = \{1,2,3\} \: \forall i \in \mathcal{M}$, dynamic matrices  $A_{ii} = \left[ \begin{smallmatrix}
 	1 & 1\\
 	0 & 1
\end{smallmatrix} \right],
A_{ij} = \left[ \begin{smallmatrix}
	0.1 & 0\\
	0.1 & 0.1
\end{smallmatrix} \right], \forall j \in \mathcal{N}_i \backslash \{i\}, \forall i \in \mathcal{M}$, input matrices $
 B_{i} = \left[ \begin{smallmatrix}
  	0\\
 	1
\end{smallmatrix} \right], \forall i \in \mathcal{M}$ and output matrices $ C_{ii} = \left[ \begin{smallmatrix}
  	1 & 0
\end{smallmatrix} \right], \forall i \in \mathcal{M}$. Each subsystem is subject to a normally distributed process noise with $\Sigma_i^w = 0.004 I$ and a chance constraint on the second state $\mathbb{P}(| [x_i]_2 | \leq 1) \geq 0.7$. As already stated in Remark \ref{rem:compactness}, we need to introduce a large, but finite constraint on unconstrained states. In this example, we introduce a constraint on the nominal state $ |[z_i]_1| \leq 50$, such that the invariant set for tracking can be determined. The weighting matrices are set to $Q_i = \left[ \begin{smallmatrix}
 	100 & 0\\
 	0 & 0.01
\end{smallmatrix} \right]$, $R_i = 1$, $T_i = 10^3$ and the prediction horizon $N = 7$. The controller ingredients are computed according to Section \ref{sec:distributed_synthesis} and we used the ADMM formulation from \cite{conte2013robust} to solve the online optimization problem. Since the noise is normally distributed, we use Remark \ref{remark:quantile} to obtain a less conservative constraint tightening.
\subsection*{Simulation results}
We carried out $1000$ Monte-Carlo simulations for $75$ closed-loop steps, starting from the initial conditions $x_1(0) = [0, 0]^\top$, $x_2(0) = [0, 0]^\top$ and $x_3(0) = [0, 0]^\top$. For the first $0 \leq k < 25$ time steps we command an admissible reference $y_{ref}^1 = [-1, 0, 1]^\top$ followed by an unreachable reference $y_{ref}^{2} = [-7, -2, 7]^\top$ for $25 \leq k < 50$ and lastly an admissible reference $y_{ref}^{3} = [0,0,0]^\top$ for $k \geq 50$. 
\begin{figure}[H]
  \centering
  \includegraphics[width=0.9\linewidth]{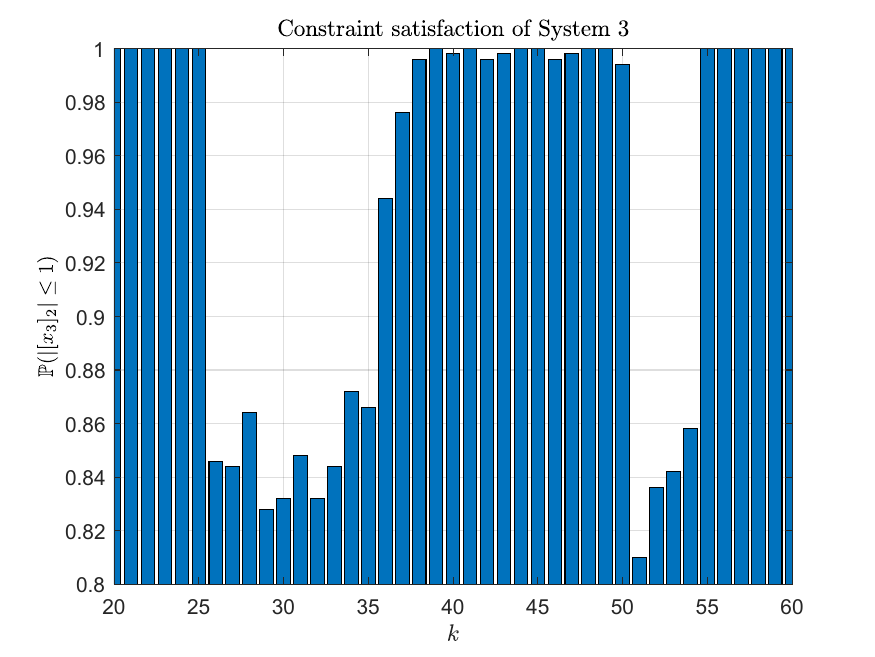}
  \caption{Constraint satisfaction of system 3}
  \label{fig:constraint}
\end{figure}
In Figure \ref{fig:output} it can be seen that the references $y_{ref}^{1}$ and $y_{ref}^{3}$ can be tracked in expectation, while the reference $y_{ref}^{2}$ is unreachable. However, the output converges to an admissible steady-state $\hat{y}_s = [-5.57, -1.45, 5.95]^\top$ that minimizes the distance to the true reference $y_{ref}^2$ in expectation, i.e. the tracking cost is $\mathbb{E} \{ J_o(\hat{y}_s, y_{ref}^{2}) \} = 3451.3$.

In Figure \ref{fig:constraint} we show the empirical constraint satisfaction of system $3$, since it is the most representative for the constraint violations in our setting. Furthermore, we cropped the picture, since the constraint violations on the second state only occur when setpoints are changed. It can be seen that the largest constraint violation is $19 \%$, which verifies that the chance constraint $\mathbb{P}( | [x_i]_2 | \leq 1) \geq 0.7$ is empirically satisfied. The gap between the required constraint satisfaction ($70 \%)$ and the empirical satisfaction rate ($81 \%$) can be deduced from the conservatism of the Chebyshev mean-variance PRS \eqref{eq:PRS:state} and the block diagonal covariance matrix \eqref{eq:block_diagonal_bound}, see also \cite{mark2020stochastic}.
\subsection{Conclusion}
We presented a distributed stochastic MPC for tracking of piecewise-constant output references. The formalism of mean-variance $t$-step PRS was utilized to guarantee closed-loop chance constraint satisfaction. Recursive feasibility is established by conditioning the MPC optimization problem on feasibility and by enforcing the augmented terminal state to be in a invariant set for tracking. Furthermore, for changing setpoints the MPC output is guaranteed to converge in expectation to nominal admissible steady-states. The paper closes with a numerical example, highlighting the tracking behavior and chance constraint satisfaction.

\bibliography{lib}
\bibliographystyle{IEEEtran}

\newpage

\begin{appendix}
\subsection{Auxiliary Lemmas for closed-loop chance constraints}
\begin{lemma}[\hspace{1sp}\cite{hewing2018stochastic}]
	\label{lem:nested}
	If $\mathcal{Q}(0, \Sigma^w)$ is central convex unimodal, any convex symmetric $k$-step PRS $\mathcal{R}_k$ is also a $k-1$ step PRS.
\end{lemma}

\begin{lemma}
		\label{lem:closed_loop}
	Let $\mathcal{Q}(0, \Sigma^w)$ be central convex unimodal and $\mathcal{R}_k$ be a convex symmetric $k$-step PRS. For system \eqref{eq:real_system} under control law \eqref{eq:tube_controller} resulting from \eqref{mpc_problem}, we have that
	\begin{align*}
			\mathbb{P}(e(k) \in \mathcal{R}_{k}) \geq \mathbb{P}(e(k|0) \in \mathcal{R}_{k})
	\end{align*}
for all $k \geq 0$, conditioned on $e(0) = e(0|0) = 0$.
\end{lemma}

\begin{proof}
The proof relies on the result \cite[Theorem 3]{hewing2018stochastic}, where it is shown that a PRS $\mathcal{R}$ satisfies
\begin{align}
	\mathbb{P}(e(n| k-n) \in \mathcal{R}) \geq \mathbb{P}(e(n+1|k-n-1) \in \mathcal{R}) \label{proof:guarantee}
\end{align}
for $n = 0, \ldots, k-1$. To show that the same holds true for $n$-step PRS, we use the fact that $\mathcal{Q}(0, \Sigma^w)$ is CCU and $\mathcal{R}_n$ is convex symmetric. Thus, Lemma \ref{lem:nested} implies that the sequence of $n$-step PRS is nested
\begin{align*}
	\mathcal{R}_0 \subseteq \mathcal{R}_1 \subseteq \ldots \subseteq \mathcal{R}_n \subseteq \mathcal{R} = \lim_{n \rightarrow \infty} R_n.
\end{align*}
In view of this, any PRS $\mathcal{R}$ is an $\infty$-step PRS and thus, the guarantees from  \eqref{proof:guarantee} carry over to $n$-step PRS, that is
 \begin{align*}
 	\mathbb{P}(e(n| k-n) \in \mathcal{R}_{n+1}) \geq \mathbb{P}(e(n+1|k-n-1) \in \mathcal{R}_{n+1}),
 \end{align*}
for all $n = 0, \ldots, k-1$. 
\end{proof}

\subsection{Proof of Theorem \ref{thm:main_result}}
The proof consists of three parts. First, recursive feasibility of the MPC problem is established. Afterwards, convergence of the states trajectories to the artificial steady-states is proven. In the last part we show that the artificial steady-states converge to the optimal admissible steady-state.
\subsection*{Recursive feasibility}
The first part of the proof verifies the recursive feasibility of the proposed controller. Let $(z_s(k), v_s(k))$ be an admissible steady-state pair consistent with output $y_s(k)$ that satisfies \eqref{eq:steady_state_covariance_global}. Assume that at time $k$ a solution to \eqref{mpc_problem} exists with optimal input and state sequences $\mathcal{V}(k) = \{ v(0|k), \ldots, v(N-1|k) \}$ and $\mathcal{Z}(k) = \{ z(0|k), \ldots, z(N|k) \}$. 

Now at time $k+1$ we consider the possible suboptimal initialization due to infeasibility in Mode $1$. Thus, we shift the state and input sequences by one time step
\begin{subequations}
\begin{align}
	\tilde{z}(t|k+1) &= z(t + 1|k) \quad \forall t = 0, \ldots, N - 1 \label{eq:shifted:state}\\
	\tilde{v}(t|k+1) &= v(t + 1|k) \quad \forall t = 0, \ldots, N - 2 \label{eq:shifted:input}
\end{align}
\end{subequations}
and append the terminal controller $\tilde{v}(N|k+1) = v_s + K (z(N|k) - z_s(k))$.

 In view of feasibility at time $k$ we have that the state and input constraints \eqref{eq:global_nominal_constraints} are satisfied for all $t = 0, \ldots, N-2$. At time $t = N - 1$, the terminal constraint \eqref{eq:terminal_constraint} ensures that the augmented state $\tilde{\alpha}(N-1|k+1) = \alpha(N|k)$ lies in the terminal set. In view of the invariance property \eqref{eq:terminal_set:invariance_cond} under the terminal controller $\tilde{v}(N|k+1)$ also the successor $\tilde{\alpha}(N|k+1) \in \mathbb{Z}_{tr}$. Hence, by definition of the terminal set, in particular \eqref{eq:terminal_set:constr}, the constraints \eqref{eq:global_nominal_constraints} are verified for all future times, i.e. the MPC problem \eqref{mpc_problem} is recursively feasible.
\subsection*{Closed-loop chance constraints}
In view of recursive feasibility it is proven that the nominal constraints \eqref{eq:global_nominal_constraints} are verified for all times $k\geq 0$, i.e. the chance constraints \eqref{eq:local_chance_constraints} are verified in prediction. Closed-loop chance constraint satisfaction follows by application of Lemma \ref{lem:closed_loop}.
\subsection*{Convergence}
At time $k+1$, we distinguish between initialization of $z(0|k+1)$ in Mode $1$ ($M^1$) and Mode $2$ ($M^2$) due to \eqref{eq:initial_value}, i.e.
\begin{align}
	&\mathbb{E} \{ J^*(z(k+1), y_s(k+1)) \} \nonumber \\
	=& \mathbb{E} \{ J^*(z(k+1), y_s(k+1))\: | \: M^2 \} \:\text{Pr}(M^2) \nonumber \\
	+& \mathbb{E} \{ J^*(z(k+1), y_s(k+1)) \:|\: M^1 \} \: \text{Pr}(M^1). \label{eq:proof:seperation}
\end{align}
The first term can be evaluated w.r.t. the suboptimal (shifted) solution \eqref{eq:shifted:state} and \eqref{eq:shifted:input}
\begin{align}
	&\mathbb{E}\{ J^*(z(k+1), y_s(k+1))\: | \: M^2 \}  = J^*( z(1|k), y_s(k+1) ) \nonumber \\
	&\leq J(\tilde{z}(\cdot|k+1), \tilde{v}(\cdot|k+1), y_s(k)) \label{eq:proof:mode2_result}
\end{align}
and the shifted suboptimal cost satisfies
\begin{align*}
&J(\tilde{z}(\cdot|k+1), \tilde{v}(\cdot|k+1), y_s(k)) \nonumber\\
&= J^*(z(k), y_s(k) ) - \Vert \Delta z(0|k) \Vert_Q^2 - \Vert \Delta v(0|k) \Vert_R^2 \nonumber\\
& + \Vert \Delta z(N|k) \Vert_{Q}^2 +  \Vert \Delta z(N|k) \Vert_{K^\top R K}^2 \\
& - \Vert \Delta z(N|k) \Vert_P^2 + \Vert A_K \Delta z(N|k) \Vert_P^2 \nonumber
\\
& \overset{\eqref{eq:lyapunov_eqn}}{=} J^*(z(k), y_s(k)) - \Vert \Delta z(0|k) \Vert_Q^2 - \Vert \Delta v(0|k) \Vert_R^2.
\end{align*}
For the second term in \eqref{eq:proof:seperation} we have
\begin{align}
&\mathbb{E} \big\{ J^*(z(k+1), y_s(k+1)) \: | \: M^1 \big \} \nonumber \\
= &\mathbb{E} \big \{ J^*(x(k+1), y_s(k+1)) \:| \: M^1 \big \} \nonumber \\
\leq& J^*(z(1|k), y_s(k+1)) + L \mathbb{E} \big \{ \Vert x(k+1) - z(1|k) \Vert_2 \: | \: M^1 \big \} \nonumber \\ 
\leq& J(\tilde{z}(\cdot|k+1), \tilde{v}(\cdot|k+1), y_s(k) ) \nonumber\\ 
& \quad + L \mathbb{E} \big \{ \Vert x(k+1) - z(1|k) \Vert_2 \: | \: M^1 \big \}, \label{eq:proof:mode1_result}
\end{align}
where the first inequality follows from
\begin{multline*}
J^*(x(k+1), y_s(k+1)) \leq \\
 J^*(z(1|k), y_s(k+1)) + L \Vert x(k+1)-z(1|k) \Vert_2.
\end{multline*}
The Lipschitz constant $L$ exists if the feasible region of \eqref{mpc_problem} is bounded, which is usually the case in constrained control. Similar arguments have been used in \cite{hewing2018stochastic, mark2020stochastic}. The second inequality is due to the shifted suboptimal cost.

 Adding $L \mathbb{E} \big \{ \Vert x(k+1) - z(1|k) \Vert_2 \: | \: M^2 \big \}$ to \eqref{eq:proof:mode2_result} and substituting it together with \eqref{eq:proof:mode1_result} into \eqref{eq:proof:seperation}, we obtain
 \begin{align*}
 &\mathbb{E} \big \{ J^*(z(k+1), y_s(k+1)) \big \} \\
 \leq & J(\tilde{z}(\cdot|k+1), \tilde{v}(\cdot|k+1), y_s(k)) \\
 & \quad + L \mathbb{E} \big \{ \Vert x(k+1) - z(1|k) \Vert_2 \big \} \\ 
 =& J^*(z(k), y_s(k)) - \Vert \Delta z(0|k) \Vert_Q^2 - \Vert \Delta v(0|k) \Vert_R^2 \\ 
 & \quad +  L \mathbb{E} \big \{ \Vert x(k+1) - z(1|k) \Vert_2 \big \}.
  \end{align*} 
 Thus, the expected cost decrease is given by 
 \begin{align}
 &\mathbb{E} \big \{ J^*(z(k+1), y_s(k+1)) \big \} - J^*(z(k), y_s(k)) \nonumber \\
\leq &- \Vert \Delta z(0|k) \Vert_Q^2 - \Vert \Delta v(0|k) \Vert_R^2 \nonumber \\
 &\quad + \underbrace{L/(\sqrt{\lambda_{\text{min}}(P)})}_{= \beta} \mathbb{E} \big \{ \Vert x(k+1) - z(1|k) \Vert_P \big \}, \label{eq:proof_cost_dec}
  \end{align}
  where we used the relation $\lambda_{\text{min}}(P) \Vert x \Vert_2^2 \leq \Vert x \Vert_P^2$.
 Following the lines of \cite{hewing2018stochastic}, the latter term can be further simplified as
 \begin{align*}
  &\mathbb{E} \big \{ \Vert x(k+1) - z(1|k) \Vert_P \big \} \leq  \mathbb{E} \big \{ \Vert A_K e(k) + w(k)\Vert_P \big \} \\
  &\leq \Vert A_K e(k) \Vert_P + \mathbb{E}(\Vert w(k) \Vert_P) \\
  &\leq (1-\epsilon) \Vert e(k) \Vert_P + \mathbb{E}( \Vert w(k) \Vert_P,
 \end{align*}
  where second inequality uses
 \begin{align*}
 	\Vert A_K e(k) \Vert_P - \Vert e(k) \Vert_P \leq - \epsilon \Vert e(k) \Vert_P
 \end{align*}
 and $P$ denotes the solution to the Lyapunov equation \eqref{eq:lyapunov_eqn} for some $\epsilon > 0$. Thus, combining the latter with \eqref{eq:proof_cost_dec}, we obtain
  \begin{align*}
 &\mathbb{E} \big \{ J^*(z(k+1), y_s(k+1)) \big \} - J^*(z(k), y_s(k)) \nonumber \\
\leq &- \Vert \Delta z(0|k) \Vert_Q^2 - \Vert \Delta v(0|k) \Vert_R^2 \nonumber  - \epsilon \beta \Vert e(k) \Vert_P + \beta \mathbb{E} \big \{ \Vert w(k) \Vert_P \big \}.
  \end{align*}  
Using standard arguments, we conclude that $\displaystyle \lim_{k \rightarrow \infty} \mathbb{E} \{ z(0|k) - z_s(k) \} = 0$ and $\displaystyle \lim_{k \rightarrow \infty} \mathbb{E} \{ v(0|k) - v_s(k) \} = 0$, i.e. the states and inputs converge in expectation to a stable operating point $y_s(k) \in \bar{\mathbb{Y}}_s$, such that $(0, z_s(k), v_s(k)) \in \mathbb{Z}_{tr}$.
 
\subsection*{Optimality of the steady-state}
The previous section established that the nominal states and inputs converge to an artificial operating point $(z_s(k), v_s(k))$ in expectation. However, as $k \rightarrow \infty$ it can be shown, e.g. \cite[Lemma 1]{limon2010robust}, that $(z_s(k), v_s(k))$ converge to the optimal admissible steady-state pair $(z_s^*, v_s^*)$ that corresponds to the optimal admissible reference $y_s^*$. 

The first assertions can now be proved. Let $y_{ref} \in \bar{\mathbb{Y}}_s$, then by \cite[Lemma 1]{limon2010robust} the artificial tracking target $y_s(k)$ converges to $y_s^*$ as $k \rightarrow \infty$. Since $y_{ref}$ is admissible, $y_s^* \rightarrow y_{ref}$, which implies that $\displaystyle \lim_{k \rightarrow \infty} \mathbb{E} \{ {y}(k)\} = y_{ref} $.

The second assertion can be shown as follows. Let $y_{ref} \notin \bar{\mathbb{Y}}_s$, then, due to the terminal set for tracking \eqref{eq:terminal_constraint} and \cite[Lemma 1]{limon2010robust}, the artificial tracking target $y_s(k)$ converges to the optimal admissible operating point $y_s^* \in \bar{\mathbb{Y}}_s$ as $k \rightarrow \infty$. The cost associated with that optimal steady-state is
\begin{align*}
	J_o(y_s^*, y_{ref}) = \Vert  y_s^* - y_{ref} \Vert_T^2.
\end{align*}
Since $y_s^*$ is optimal, it follows from convexity of $J_o$ that $\forall y_s \in \bar{\mathbb{Y}}_s \backslash \{ y_s^* \} : J_o(y_s, y_{ref}) > J_o(y_s^*, y_{ref})$. Thus, $y_s^* = \tilde{y}_s$ is the unique minimizer of 
\begin{align*}
	\tilde{y}_s = \!\arg \: \!\min_{y_s \in \bar{\mathbb{Y}}_s} \Vert y_s - y_{ref} \Vert_T^2.
\end{align*}
Similar results have been reported in \cite[Theorem 1]{limon2010robust}. $\quad\blacksquare$

\end{appendix}

\vspace{12pt}
\end{document}